\documentclass[10pt,a4paper]{article}
\usepackage[utf8]{inputenc}
\usepackage[hmargin=3cm,vmargin=2.6cm]{geometry}
\usepackage{amsmath,amssymb,amsthm}
\usepackage{booktabs,array,multicol}
\usepackage{microtype}
\usepackage{cite}
\usepackage[font=small,skip=4pt]{caption}
\usepackage{float}
\usepackage{placeins}
\usepackage[hidelinks]{hyperref}

\newtheorem{theorem}{Theorem}[section]
\newtheorem{proposition}[theorem]{Proposition}
\newtheorem{lemma}[theorem]{Lemma}
\newtheorem{corollary}[theorem]{Corollary}
\theoremstyle{definition}
\newtheorem{definition}[theorem]{Definition}

\newcommand{\Z}{\mathbb{Z}}
\newcommand{\Q}{\mathbb{Q}}
\newcommand{\C}{\mathbb{C}}
\DeclareMathOperator{\PAF}{PAF}
\DeclareMathOperator{\PSD}{PSD}

\newcommand{\Dbar}{\leavevmode\rlap{\kern-0.005em\raisebox{0.32ex}{\rule{0.25em}{0.06ex}}}D}
\newcommand{\Dj}{\Dbar okovi\'c}

\title{An exact compression-and-lifting procedure and \\ D-optimal designs of orders 198 and 170}
\author{Stelios D. Georgiou\thanks{Corresponding author: stelios.georgiou@rmit.edu.au} \ and Stella Stylianou\\
\small School of Science, RMIT University, Melbourne, Australia
}
\hypersetup{pdftitle={An exact compression-and-lifting procedure and D-optimal designs of orders 198 and 170},pdfauthor={Stelios D. Georgiou and Stella Stylianou}}
\date{}
\begin{document}
\maketitle

\begin{abstract}
We develop a compression-and-lifting procedure for constructing D-optimal complementary pairs on finite abelian groups of odd order. The procedure separates the Fourier conditions preserved by compression from those needed to recover the original binary sequences. For groups with a cyclic factor of order nine, integer matching conditions separate compatible classes of local sign assignments. Their exact sizes can be computed before full sequences are generated, allowing searches to discard incompatible pairs and select manageable classes. A bound on cosets on which a sequence is constant also excludes certain automorphisms as fixing symmetries. Applying the method yields supplementary difference sets with parameters $(99;43,42;36)$ and $(85;39,34;31)$, resolving two previously undecided parameter sets. These give the first D-optimal matrix of order $198$ and constructions with new parameters at the known order $170$. For the successful length-$99$ compression, the exact class condition removes $99.9888\%$ of candidate pairs while preserving every D-optimal lift. A reflection-invariant construction at length $85$ gives seven supplementary difference set classes and fourteen matrix classes under signed row and column permutations, paired by transposition. The class reductions guide computation; finding suitable compressions and matching their lifts remain search problems.
\end{abstract}


\section{Introduction}
For a $\pm1$ matrix of order $2v$, with $v$ odd, the Ehlich--Wojtas bound is
\begin{equation}\label{eq:EW}
 |\det M|\le 2^v(2v-1)(v-1)^{v-1}.
\end{equation}
A matrix attaining this bound is called a D-optimal design \cite{Eh64,Wo64}. Such matrices solve the maximal determinant problem at their orders and give saturated two-level designs in statistics; see \cite{KO07,BEHO21}. A standard construction uses two binary sequences, with entries in $\{\pm1\}$, whose periodic autocorrelations sum to $2$ at every nonzero shift. Their negative positions form a supplementary difference set (SDS).

Compression replaces each sequence by its sums over cosets of a subgroup. It preserves the Fourier data at characters trivial on that subgroup and is an established tool for constructing complementary pairs \cite{DK15a,DK15b,DK18}. The missing data must be recovered by assigning signs within each coset, called a fibre. This lifting step can dominate the computation. The successful compression at length $99$ in this paper leaves $3^{49}$ pairs of local sign assignments. Enumerating and comparing every pair is impractical. We address how to count, separate and match these assignments before performing the full autocorrelation tests.

Our construction procedure starts with a character-channel criterion on any finite abelian group of odd order. Each channel collects the Fourier data associated with one character of the compression subgroup. The criterion identifies the equations still needed after a complementary compressed pair has been chosen. It gives a complete test for recovering the original sequences from local sign patterns. The Fourier decomposition is standard \cite{DK15a,DK18}. We use it to derive exact counts and matching conditions for these patterns, so that searches can reject incompatible pairs before generating full sequences.

Write $\Z_n$ for the integers modulo $n$. For groups $\Z_9\times H$, where $H$ is abelian of odd order $w>1$, each three-entry fibre with mixed signs has one minority sign and three possible positions for it. Within each compressed row, a row group consists of assignments with the same signed counts of these positions. The counts determine a length-nine compression, whose relevant spectral values are encoded by a key of three integers. Generating functions give the number of lifts with each key. At most $(w+1)^6$ combinations of row groups are needed per side, without enumerating the exponentially many full lifts. Within compatible classes, meet-in-the-middle matching generates the two lists separately and looks for complementary autocorrelation values. This avoids comparing their entire Cartesian product. Another branch imposes reflection symmetry on five-entry fibres and derives an exact integer test for the remaining conditions. Both branches give researchers constructing difference families explicit rules for reducing and organising the lifting computation.

A structural bound limits the number of fibres on which all signs agree. It excludes a family of automorphisms that could otherwise be proposed as fixing symmetries of a block. Under a coprimality condition, the exclusion also holds after translating the block. This gives a way to reject impossible symmetry assumptions before a search begins and complements earlier multiplier methods \cite{KSWX97,GK99}. Related compression searches for Legendre pairs, whose binary sequences have combined autocorrelation $-2$ away from zero, use prescribed row and column sums \cite{TKBG21}. Further compression properties are studied in \cite{KGG25,KGGG26}. Our class counts and matching targets concern the different autocorrelation value $2$ required for D-optimal pairs.

Applying the method gives SDSs with parameters $(99;43,42;36)$ and $(85;39,34;31)$, both undecided in \cite{Do22}. The former produces a D-optimal matrix of order $198$, the smallest unresolved order recorded in \cite{DK15b,Do22}. The latter settles a parameter set at order $170$, where D-optimal matrices were already known \cite{GK99}. The recorded length-$85$ collection contains seven SDS classes. Its fourteen matrix representatives are pairwise inequivalent under signed row and column permutations and form seven pairs under transposition. This classification concerns the recorded collection, not all matrices of order $170$.

For the successful length-$99$ compression, the exact key reduces $3^{49}\simeq2.3930\cdot10^{23}$ lift pairs to $2.6710\cdot10^{19}$, a factor of about $8{,}959$, without losing any D-optimal lift. The reported discovery cost is approximately $3{,}200$ CPU-hours, meaning processor time summed over all workers. The reduction in candidates is not a measured runtime speed-up. Computation is still needed to find compressed pairs and match their lifts, with costs affected by storage and pruning. The discovery thresholds select a subset of classes, so the reported run is not globally exhaustive.

Section~\ref{sec:prelim} recalls the standard identities. Section~\ref{sec:channels} gives the lifting criterion and structural bounds. Sections~\ref{sec:lift99} and~\ref{sec:lift85} develop the two constructive branches, which are applied and verified in Section~\ref{sec:designs}. Section~\ref{sec:open} discusses their scope and further applications.

\section{Preliminaries}\label{sec:prelim}

\subsection{Autocorrelation on finite abelian groups}
Let $G$ be a finite abelian group, written additively, and let $f\colon G\to\C$. A character is a homomorphism from $G$ to the unit circle; $\widehat G$ denotes the group of these characters, whose trivial element is the constant character $1$. Write $\zeta_n=\exp(2\pi i/n)$, where $i^2=-1$. On $\Z_n$ we index characters by $\chi_s(j)=\zeta_n^{sj}$. Define the periodic autocorrelation function (PAF), Fourier transform and power spectral density (PSD) by
\[
\PAF_f(k)=\sum_{g\in G}f(g)\overline{f(g+k)},\qquad \hat f(\chi)=\sum_{g\in G}f(g)\chi(g),\qquad \PSD_f(\chi)=|\hat f(\chi)|^2 .
\]
The energy of $f$ is $\PAF_f(0)=\sum_g|f(g)|^2$. A pair $(f_1,f_2)$ is \emph{complementary with constants} $(P,c)$ if $\PAF_{f_1}(k)+\PAF_{f_2}(k)$ equals $P$ for $k=0$ and $c$ for every $k\ne0$.

The standard Fourier identity below converts the required autocorrelations
into spectral tests \cite{DK15a}. These tests later provide both matching
conditions and bounds for rejecting partial assignments.
\begin{lemma}\label{lem:fourier}
The pair $(f_1,f_2)$ is complementary with constants $(P,c)$ if and only if $\PSD_{f_1}(\chi)+\PSD_{f_2}(\chi)$ equals $P-c$ for every nontrivial character $\chi$ and $P+(|G|-1)c$ for the trivial one.
\end{lemma}

For a subgroup $K\le G$ of order $m$, each coset is a \emph{fibre}; it is \emph{homogeneous} for a binary sequence when all its signs agree. Otherwise it is \emph{nonhomogeneous}. The \emph{compression} of $f$ along $K$ is $f_K\colon G/K\to\C$, defined by $f_K(g+K)=\sum_{h\in K}f(g+h)$. A binary \emph{lift} of $U$ is a function $a\colon G\to\{\pm1\}$ with $a_K=U$.

The following compression identities retain exactly the Fourier data needed
for the smaller search \cite[Theorems 3 and 4]{DK15a},
\cite[Theorem 1]{DK15b}; see also \cite[Section~7]{DK18} for the group formulation.
\begin{theorem}\label{thm:compress}
For every character $\psi$ of $G/K$, $\PSD_{f_K}(\psi)=\PSD_f(\psi\circ\pi)$, where $\pi\colon G\to G/K$ is the quotient map, and $\PAF_{f_K}(k+K)=\sum_{h\in K}\PAF_f(k+h)$. In particular, if $(f_1,f_2)$ is complementary with constants $(P,c)$, then the compressed pair is complementary with constants $(P+(m-1)c,\,mc)$.
\end{theorem}

\subsection{D-optimal pairs}
Let $v$ be odd and $G$ abelian of order $v$. A pair $(a,b)$ of functions $G \to \{\pm1\}$ is a \emph{D-optimal pair} if it is complementary with constants $(2v, 2)$. Lemma~\ref{lem:fourier} gives the necessary individual bounds
\[
 \PSD_a(\chi)\le 2v-2,\qquad \PSD_b(\chi)\le 2v-2
 \qquad(\chi\ne1).
\]
These bounds can discard a proposed lift without generating its partner.

The standard group-matrix construction turns such a pair into a matrix
attaining \eqref{eq:EW} \cite{Eh64,KO07}.
\begin{lemma}\label{lem:twocirc}
Let $(a,b)$ be a D-optimal pair on $G$, and let $A=(a(h-g))_{g,h\in G}$ and $B=(b(h-g))_{g,h\in G}$. Then
\[
M=\begin{pmatrix}A&B\\-B^{\top}&A^{\top}\end{pmatrix}
\quad\text{satisfies}\quad MM^{\top}=I_2\otimes\bigl((2v-2)I_v+2J_v\bigr),
\]
where $I_v$ is the identity matrix and $J_v$ is the all-ones matrix of order $v$, and $|\det M|=2^{v}(2v-1)(v-1)^{v-1}$.
\end{lemma}

For $G=\Z_v$, $A$ and $B$ are the circulants with first rows $a$ and $b$. Let $X=a^{-1}(-1)$ and $Y=b^{-1}(-1)$, with $|X|=r$ and $|Y|=s$. Write $\lambda_X(k)$ for the number of ordered pairs $(x,x')\in X^2$ with $x-x'=k$, and define $\lambda_Y$ similarly. The blocks form a supplementary difference set, or difference family, $(v;r,s;\lambda)$ when $\lambda_X(k)+\lambda_Y(k)=\lambda$ for every $k\ne0$. Since $\PAF_a(k)=v-4r+4\lambda_X(k)$, this is equivalent to D-optimality when $\lambda=r+s-(v-1)/2$.

For fixed ordered block sizes, \emph{SDS equivalence} is generated by independent translations, a common automorphism of $G$, and reversal $Z\mapsto-Z$ of either block \cite[Section~2]{Do11}. On $\Z_v$, an automorphism is multiplication by a unit modulo $v$. Reversing both blocks is already a common automorphism. Complementation and interchange normalise the parameters but add no identifications between the fixed ordered sizes here. Matrix equivalence means independent row and column permutations and sign changes. Transposition is allowed only when explicitly stated.

The trivial character gives
\begin{equation}\label{eq:sums}
 (\textstyle\sum a)^2+(\textstyle\sum b)^2=4v-2.
\end{equation}
At $v=99$, its only representation by two nonnegative odd squares is $394=13^2+15^2$. Complementing and interchanging blocks therefore gives sums $13,15$ and parameters $(99;43,42;36)$. At $v=85$, the parameters studied here give sums $7,17$, with $7^2+17^2=338$.

Compression energy fixes the numbers of the two fibre types at length $99$
\cite{DK15a}. These counts determine the size of the later lifting problem.
\begin{corollary}\label{cor:nu}
Let $(a,b)$ be a D-optimal pair on an abelian group of order $99$ and $K$ a subgroup of order $3$. The compressions $a_K$ and $b_K$ have entries $\pm1$ and $\pm3$, with exactly $\nu_1=49$ entries $\pm1$ and $\nu_3=17$ entries $\pm3$ in the two sequences together.
\end{corollary}

By Theorem~\ref{thm:compress}, the compressions of a D-optimal pair for $v=99$ to length $d=3,9,11,33$ have odd entries of absolute value at most $99/d$, keep the sums $13$ and $15$, and are complementary with constants
\begin{equation}\label{eq:const}
(P_d,c_d)=(262,66),\ (218,22),\ (214,18),\ (202,6)\qquad\text{for } d=3,\,9,\,11,\,33.
\end{equation}

\section{Compression and exact lifting}\label{sec:channels}

Compression retains the sums on subgroup cosets but loses the positions of
the signs. The standard character decomposition \cite{DK15a,DK18} separates
the retained Fourier data from the missing data. We express this separation
as a lifting criterion, then use local sign patterns to enforce its
remaining conditions.

Let $G$ be abelian of odd order $v$, let $K\le G$ have order $m$, and put
$Q=G/K$, $d=|Q|$, with quotient map $\pi$. Let $a,b\colon G\to\{\pm1\}$,
and write $U_a=a_K$, $U_b=b_K$ for their compressions. For each character
$\rho\in\widehat K$, choose an extension $\eta_\rho\in\widehat G$, taking
$\eta_1=1$.  Such extensions exist for finite abelian groups.  Define
\begin{equation}\label{eq:channels}
 F_{a,\rho}(q)=\sum_{g\in\pi^{-1}(q)}a(g)\eta_\rho(g),
 \qquad q\in Q.
\end{equation}
A \emph{character channel} is the function in \eqref{eq:channels}; define $F_{b,\rho}$ analogously. The trivial channel is the ordinary compression $F_{a,1}=U_a$.
For $\psi\in\widehat Q$, direct summation gives
\begin{equation}\label{eq:channel_transform}
 \widehat{F_{a,\rho}}(\psi)
   =\hat a\bigl(\eta_\rho(\psi\circ\pi)\bigr).
\end{equation}
Thus each channel contains precisely the characters of $G$ restricting
to $\rho$ on $K$. The following criterion makes the remaining lifting
conditions explicit. Once the compressed pair is fixed, only its
nontrivial channels need to be matched.

\begin{theorem}\label{thm:channels}
Let $a,b\colon G\to\{\pm1\}$ satisfy
$(\sum_Ga)^2+(\sum_Gb)^2=4v-2$.  Then $(a,b)$ is a D-optimal pair if and
only if
\begin{align}
 \PAF_{U_a}(q)+\PAF_{U_b}(q)&=2m
       &&(q\ne0),\label{eq:compressed_channel}\\
 \PAF_{F_{a,\rho}}+\PAF_{F_{b,\rho}}
       &=(2v-2)\delta_0
       &&(\rho\ne1),\label{eq:other_channels}
\end{align}
where $\delta_0$ is one at the identity of $Q$ and zero elsewhere.
In this case the combined compression energy is $2v+2(m-1)$.
Extensions may be chosen with
$\eta_{\bar\rho}=\overline{\eta_\rho}$; with this choice it suffices to
check one character from each pair $\{\rho,\bar\rho\}$ in
\eqref{eq:other_channels}.
\end{theorem}

\begin{proof}
The cosets $\eta_\rho K^\perp$, for $\rho\in\widehat K$, partition
$\widehat G$, where $K^\perp$ is the set of characters trivial on $K$.
If $(a,b)$ is D-optimal, its combined PSD is $2v-2$ at every nontrivial
character and $4v-2$ at the trivial character.  Compression gives
\eqref{eq:compressed_channel} and energy $2v+2(m-1)$.
For $\rho\ne1$ every character in $\eta_\rho K^\perp$ is nontrivial,
so \eqref{eq:channel_transform} and Fourier inversion give
\eqref{eq:other_channels}.

Conversely, summing the compressed autocorrelations over $Q$ gives
$4v-2$.  Hence \eqref{eq:compressed_channel} forces their combined
value at zero to be
$4v-2-2m(d-1)=2v+2(m-1)$.  The compression channel therefore has
combined PSD $2v-2$ at every nontrivial character of $Q$.
Equation~\eqref{eq:other_channels} supplies the same value on each
remaining character coset.  Together with the assumed sums, these are
exactly the Fourier conditions for a D-optimal pair.
Finally, the real-valuedness of $a,b$ gives
$F_{a,\bar\rho}=\overline{F_{a,\rho}}$ and the analogous identity for
$b$.  Since $m$ is odd, all nontrivial characters of $K$ occur in distinct
conjugate pairs, proving the last assertion.
\end{proof}

For a subgroup of order three, only one nontrivial channel needs to be
checked, because the other is its conjugate. The next consequence also
counts the local choices. It gives the size of the unrestricted lifting
problem before class matching or symmetry restrictions are applied.

\begin{corollary}\label{thm:subgroup_three}
Suppose $|K|=3$.  Let $U_a,U_b\colon Q\to\{\pm1,\pm3\}$ satisfy
\eqref{eq:compressed_channel} and
$(\sum_QU_a)^2+(\sum_QU_b)^2=4v-2$.  Across the two compressions there
are exactly
\[
 \nu_1=\frac{v-1}{2},\qquad \nu_3=\frac{v+3}{6}
\]
entries of absolute value one and three, respectively.  Consequently
there are exactly $3^{(v-1)/2}$ ordered pairs of sign lifts.
Choose the extension $\eta$ used in \eqref{eq:channels} for a nontrivial character of $K$, and write
$F_a=F_{a,\eta|_K}$.  A lift is D-optimal if and only if
\begin{equation}\label{eq:three_lift}
 \PAF_{F_a}+\PAF_{F_b}=(2v-2)\delta_0.
\end{equation}
On a fibre with $|U_a(q)|=3$, $F_a(q)=0$.  On a fibre with
$|U_a(q)|=1$, let $g_q$ be the unique position with sign $-U_a(q)$.
Then
\begin{equation}\label{eq:minority_phase}
 F_a(q)=-2U_a(q)\eta(g_q).
\end{equation}
\end{corollary}

\begin{proof}
The compression energy is $2v+4$, so
$\nu_1+\nu_3=2v/3$ and $\nu_1+9\nu_3=2v+4$, yielding the stated
counts.  A homogeneous fibre has one sign pattern; every other fibre
has three, according to the position of its minority sign.  The
nontrivial character sums to zero on each coset of $K$, which gives
both channel formulas.  There are only two nontrivial characters of
$K$, and they are conjugate, so Theorem~\ref{thm:channels} reduces their
conditions to \eqref{eq:three_lift}.
\end{proof}

For the cyclic prime-power factor used below, the channels have a
convenient explicit form. Let $p$ be an odd prime and $k\ge2$. Write
$G=\Z_{p^k}\times H$ with $H$ finite abelian of odd order, put $e=p^{k-1}$,
and write $j=x+ey$, where $0\le x<e$ and $y\in\Z_p$. Indexing the
characters of $K=e\Z_{p^k}\times\{0\}$ by $r\in\Z_p$, equation
\eqref{eq:channels} becomes
\begin{equation}\label{eq:explicit_channels}
 F_{a,r}(x,h)=\zeta_{p^k}^{rx}
       \sum_{y\in\Z_p}a(x+ey,h)\zeta_p^{ry},
 \qquad
 \widehat{F_{a,r}}(u,\psi)=\hat a(r+pu,\psi).
\end{equation}
Here $u\in\Z_e$, $\psi\in\widehat H$, and $\hat a(s,\psi)$ uses the
character $(j,h)\mapsto\zeta_{p^k}^{sj}\psi(h)$. The channel $r=0$ is the ordinary compression, and the $p$ channels cover
every character exactly once.  Section~\ref{sec:lift99} specialises
these channels to $p=3$, $k=2$, and uses the minority-sign positions to
group the lifts before matching their remaining correlation conditions.

The compression energy limits how many fibres can have constant signs.
This bound can exclude a proposed symmetry without enumerating blocks.
It also shows that no nonzero translation fixes a block when $v>3$;
equivalently, each block has a trivial translation stabiliser.

\begin{proposition}\label{prop:homogeneous}
Let $(a,b)$ be a D-optimal pair on an abelian group of odd order
$v=md$, and compress along a subgroup of odd order $m>1$.
If $L$ counts the entries of absolute value $m$ in the two
compressions together, then
\begin{equation}\label{eq:homogeneous_bound}
 L\le\left\lfloor\frac{2(d+1)}{m+1}\right\rfloor.
\end{equation}
If $v>3$, neither block has a nontrivial translation stabiliser.
\end{proposition}

\begin{proof}
All $2d$ compressed entries are odd, so their combined energy is at
least $2d+L(m^2-1)$.  Compression gives energy $2md+2(m-1)$,
and rearranging proves \eqref{eq:homogeneous_bound}.
If a block were fixed by a nonzero translation, let $K$ be the cyclic
subgroup generated by that translation and write $m=|K|$, $d=v/m$.
All $d$ fibres of that block would be homogeneous, forcing
$d(m+1)\le2(d+1)$, or $(m-1)d\le2$.
Since $m>1$ and $m,d$ are odd, this is possible only for $m=3,d=1$.
\end{proof}

A \emph{fixing multiplier} is an automorphism $h$ that leaves a block $Z$
unchanged; fixing up to translation means $hZ=Z+t$ for some $t\in G$.
The next result applies the fibre bound to a prime-power factor. It tells
a search which multiplier restrictions cannot produce a D-optimal pair.

\begin{proposition}\label{prop:multiplier}
Let $p$ be an odd prime, $k\ge2$, and $H$ an abelian group of odd
order, with $v=p^k|H|>9$.  On $G=\Z_{p^k}\times H$ put
\[
 h(j,z)=((1+p^{k-1})j,z).
\]
Neither block of a D-optimal pair is fixed by a nonidentity element of
$\langle h\rangle$.  If $p\nmid|H|$, the same exclusion holds for
fixing up to translation.
\end{proposition}

\begin{proof}
The automorphism $h$ has order $p$.  Compress along
$K=p^{k-1}\Z_{p^k}\times\{0\}$ and put $d=p^{k-1}|H|$.
Whenever $p\nmid j$, the $h$-orbit of $(j,z)$ is its entire coset
of $K$.  A fixed block would therefore give at least
$(p-1)d/p$ homogeneous fibres.  This exceeds the bound
$2(d+1)/(p+1)$ precisely when
\[
 (p^2-2p-1)d>2p.
\]
For $p=3$ this follows from $d>3$, and for $p\ge5$ from $d\ge p$.
The assumption $v>9$ covers the former condition.
Every nonidentity power of $h$ has the same orbits.

Suppose now that $p\nmid|H|$ and $hZ=Z+t$ for one block $Z$.
Its translation stabiliser is trivial by
Proposition~\ref{prop:homogeneous}, so iteration gives
$(1+h+\cdots+h^{p-1})t=0$.
The binomial identity
$(1+p^{k-1})^i\equiv1+ip^{k-1}\pmod{p^k}$ shows that this
endomorphism is $pI$.  Consequently
\[
 t\in\ker(pI)=p^{k-1}\Z_{p^k}\times\{0\}=(h-I)G.
\]
Choose $c$ with $(h-I)c=-t$.  Then $h(Z+c)=Z+c$, contradicting
the first part.  For every nonidentity power of $h$, the same
calculation gives $pI$ and the same image $(h-I)G$.
\end{proof}

In particular, on $\Z_9\times H$ with $|H|>1$, multiplication by
$4$ or $7$ in the first factor cannot fix either block; when
$3\nmid|H|$, translating the block does not remove the obstruction.
At length $99$ this excludes the order-three fixing multipliers
$67$ and $34$.  It is a restriction on possible fixing symmetries,
and does not exclude other structured ways of organising the search.

\section{Exact digit classes and matching}\label{sec:lift99}
\subsection{Local lifts and channels}
We give the digit formulas in the coordinates used for the length-$99$ application; Proposition~\ref{prop:general9} establishes their extension to $\Z_9\times H$ and bounds the class-counting work.

We identify $\Z_{99}$ with $\Z_9\times\Z_{11}$ through $i\mapsto(i\bmod 9,\,i\bmod 11)$, and write elements of $\Z_9$ as $j=x+3y$ with $x\in\{0,1,2\}$ and $y\in\Z_3$. A function $a\colon\Z_{99}\to\{\pm1\}$ is then a $9\times11$ array $a(j,z)$. Its compressions along the subgroups of orders $11$, $9$ and $3$ are the row sums $A_9(j)=\sum_z a(j,z)$, the column sums $A_{11}(z)=\sum_j a(j,z)$, and the $3\times11$ array
\[
U_a(x,z)=\sum_{y\in\Z_3}a(x+3y,z),
\]
a function on $\Z_3\times\Z_{11}\cong\Z_{33}$. Further $A_3(x)=\sum_zU_a(x,z)$. For $b$ we write $B_9$, $B_{11}$, $U_b$ and $B_3$. The characters of $\Z_9\times\Z_{11}$ are $\chi_{s,t}(j,z)=\zeta_9^{sj}\zeta_{11}^{tz}$, and we write $\hat a(s,t)$ for $\hat a(\chi_{s,t})$.

The compression identity \cite{DK15a} identifies which arrays determine
each Fourier value. This prevents retesting conditions already imposed by
the smaller compressions and locates the equations still needed in the lift.
\begin{lemma}\label{lem:split}
The value $\hat a(s,t)$ depends only on $A_9$ if $t=0$, only on $A_{11}$ if $s=0$, and only on $U_a$ if $3\mid s$. Of the $98$ nontrivial characters, $8$ have $t=0$, $10$ have $s=0$, $20$ have $s\in\{3,6\}$ and $t\ne0$, and the remaining $60$ have $s$ of order $9$ and $t\ne0$.
\end{lemma}

If $|U_a(x,z)|=3$, the three values $a(x+3y,z)$, $y\in\Z_3$, are equal. If $|U_a(x,z)|=1$, exactly one equals $-U_a(x,z)$; its index is the \emph{digit} $\rho_a(x,z)\in\Z_3$. Thus $U_a$ and its digits determine $a$. For a compression $U$, put $S(U)=\{(x,z):|U(x,z)|=1\}$ and $\nu_1(U)=|S(U)|$. Its lifts correspond to the $3^{\nu_1(U)}$ digit assignments on $S(U)$.

\begin{definition}\label{def:twist}
For a lift $a$ of $U_a$ define
\[
 F_a(x,z)=\zeta_9^x\sum_{y\in\Z_3}a(x+3y,z)\zeta_3^y
 =\begin{cases}
 -2U_a(x,z)\zeta_9^{x+3\rho_a(x,z)},&|U_a(x,z)|=1,\\
 0,&|U_a(x,z)|=3.
 \end{cases}
\]
\end{definition}
The displayed formula follows by summing the three cube roots of unity. Direct expansion gives
\begin{equation}\label{eq:channel99}
 \widehat F_a(u,t)=\widehat a(1+3u,t),\qquad (u,t)\in\Z_3\times\Z_{11}.
\end{equation}
The function $F_a$ is one of the two nontrivial subgroup channels; the other may be taken as its conjugate.

The order-three criterion now gives an exact test at length $99$.
The compressed pair supplies the retained correlations, while the digit
channel tests all remaining conditions on the local sign choices.

\begin{theorem}\label{thm:twolevel}
Let $a,b\colon\Z_{99}\to\{\pm1\}$ have sums $13$ and $15$. They form a D-optimal pair if and only if
\begin{enumerate}
\item[\textup{(i)}] $\PAF_{U_a}(k)+\PAF_{U_b}(k)=6$ for every nonzero $k\in\Z_3\times\Z_{11}$;
\item[\textup{(ii)}] $\PAF_{F_a}(k)+\PAF_{F_b}(k)=0$ for every nonzero $k\in\Z_3\times\Z_{11}$.
\end{enumerate}
In that case the combined energies at zero are $202$ and $196$, respectively.
\end{theorem}
\begin{proof}
Condition (i) and the sums give compressed energy $394-32\cdot6=202$. Since the compressed entries are $\pm1$ or $\pm3$, there are $49$ entries of absolute value one and $17$ of absolute value three. Thus the combined energy of $F_a,F_b$ is $4\cdot49=196$. The result now follows from the order-three channel criterion and \eqref{eq:channel99}.
\end{proof}

\subsection{Lift classes}
Fix a row $x$ of $U_a$. Let $p_x$ and $m_x$ be the numbers of entries $+1$ and $-1$ in that row, define the \emph{signed digit sums}
\[
g_{x,y}=\sum_{z:\ |U_a(x,z)|=1,\ \rho_a(x,z)=y}U_a(x,z)\qquad(y\in\Z_3),
\]
and put $c_x=\sum_{|U_a(x,z)|=3}U_a(x,z)/3+\sum_{|U_a(x,z)|=1}U_a(x,z)$.

A \emph{row digit group} consists of the assignments with the same signed
digit sums. The following formula counts each group and determines its
contribution to the length-nine compression.

\begin{lemma}\label{lem:digitsums}
For every lift, $A_9(x+3y)=c_x-2g_{x,y}$. The number of digit assignments of row $x$ with given signed digit sums $g_x=(g_{x,0},g_{x,1},g_{x,2})$ is the coefficient of $w_0^{g_{x,0}}w_1^{g_{x,1}}w_2^{g_{x,2}}$ in $(w_0+w_1+w_2)^{p_x}(w_0^{-1}+w_1^{-1}+w_2^{-1})^{m_x}$, where $w_0,w_1,w_2$ are formal variables.
\end{lemma}

\begin{proof}
Row $x+3y$ of the array receives $U_a/3$ from every entry $\pm3$ of row $x$ of $U_a$, and $\sigma$ or $-\sigma$ from every entry $\sigma=\pm1$, according as its digit differs from $y$ or equals $y$. The count follows by choosing the digit of each entry independently.
\end{proof}

Thus the 9-compression of a lift depends only on one triple of signed digit sums per row, and these take few values ($108$, $61$ and $46$ in the rows of $U_a$ for the design of Section~\ref{sec:designs}).

The missing conditions with $t=0$ depend only on the length-nine
compressions. The next proposition encodes them in three integers, called
a \emph{spectral key}. Matching these keys avoids numerical equality
tests and discards incompatible classes before full lifts are generated.

\begin{proposition}\label{prop:key}
For an integer sequence $A$ on $\Z_9$ let $P_k=\PAF_A(k)$ and
\[
\kappa(A)=(P_0-P_3,\;P_1-P_4,\;P_2-P_4)\in\Z^3 .
\]
For integer sequences $A$ and $B$ on $\Z_9$ and an integer $\beta$, $\PSD_A(s)+\PSD_B(s)=\beta$ for every $s$ of order $9$ if and only if $\kappa(A)+\kappa(B)=(\beta,0,0)$.
\end{proposition}

\begin{proof}
Let $\gamma_k=\zeta_9^k+\zeta_9^{-k}$. As $P_k=P_{-k}$, $\PSD_A(1)=P_0+P_1\gamma_1+P_2\gamma_2+P_3\gamma_3+P_4\gamma_4$. Here $\gamma_3=-1$, and $\gamma_1+\gamma_2+\gamma_4=0$ by the identity $1+\zeta_3+\zeta_3^2=0$. Hence
\[
\PSD_A(1)=(P_0-P_3)+(P_1-P_4)\,\gamma_1+(P_2-P_4)\,\gamma_2 .
\]
The number $\gamma_1=2\cos(2\pi/9)$ satisfies the irreducible polynomial $X^3-3X+1$ over $\Q$ (its only possible rational roots, $\pm1$, are not roots). Since $\gamma_2=\gamma_1^2-2$, the numbers $1,\gamma_1,\gamma_2$ are a basis of $\Q(\zeta_9+\zeta_9^{-1})$. So $\PSD_A(1)+\PSD_B(1)=\beta$ if and only if $\kappa(A)+\kappa(B)=(\beta,0,0)$. The values at the other $s$ of order $9$ are Galois conjugates of the value at $s=1$, and $\beta$ is rational.
\end{proof}

The standard minimal-polynomial test gives the same type of key at other
compression lengths. The following formulation explains how to obtain an
exact integer test when the three-coordinate form above is unavailable.

\begin{proposition}\label{prop:cyclotomic_key}
Let $A,B$ be integer sequences on $\Z_m$, where $m>1$, and put
\[
 R_A(z)=\sum_{k=0}^{m-1}\PAF_A(k)z^k.
\]
Define $R_B$ similarly. Let $K_m(A)$ be the coefficient vector of the remainder of $R_A$ modulo the cyclotomic polynomial $\Phi_m$, whose roots are the primitive $m$th roots of unity. Order the coefficients by increasing powers and pad with zeros to length $\varphi(m)$, where $\varphi$ is Euler's totient function. For an integer $\beta$,
\[
\begin{aligned}
 &\PSD_A(s)+\PSD_B(s)=\beta\quad\text{for every character of order }m\\
 &\hspace{25mm}\Longleftrightarrow\quad K_m(A)+K_m(B)=(\beta,0,\ldots,0).
\end{aligned}
\]
For $m>2$ the same test can be represented in a rational basis of the real cyclotomic field $\Q(\zeta_m+\zeta_m^{-1})$, using $\varphi(m)/2$ coordinates; a fixed common denominator makes these integer keys. At $m=9$ it is the three-coordinate key of Proposition~\ref{prop:key}.
\end{proposition}

\begin{proof}
At $\zeta=\zeta_m$, $R_A(\zeta)=\PSD_A(1)$. Its minimal polynomial is the monic polynomial $\Phi_m\in\Z[z]$, so $R_A+R_B-\beta$ vanishes at $\zeta$ if and only if its remainder modulo $\Phi_m$ is zero. The other primitive roots are its Galois conjugates. Autocorrelation symmetry makes the values real; for $m>2$ they lie in the degree-$\varphi(m)/2$ real subfield. Clearing the fixed basis denominators gives the integer form.
\end{proof}

Applied to $A_9$ and $B_9$, Proposition~\ref{prop:key} is the part of condition (ii) at $t=0$, while the characters of order $3$ with $t=0$ are already covered by (i). We call $\kappa(A_9)$ the \emph{key} of a lift of $U_a$.

\begin{definition}\label{def:class}
Fix $(U_a,U_b)$ satisfying Theorem~\ref{thm:twolevel}\,\textup{(i)} with sums $13,15$. Put $T=(196,0,0)$. The class indexed by $\kappa\in\Z^3$ consists of pairs whose $a$-lift has key $\kappa$ and whose $b$-lift has key $T-\kappa$. Write $N_A(\kappa)$ and $N_B(\kappa)$ for the respective numbers of lifts. A class is \emph{matched} if both counts are positive.
\end{definition}

Triples of row digit groups partition all lifts. Their multiplicities
therefore give exact class sizes without generating the sequences. These
counts identify compatible classes and measure the candidate volume left
by the key, providing the data needed to choose manageable searches.

\begin{proposition}\label{prop:classcounts}
For side $S\in\{A,B\}$, let $N_{S,x}(g_x)$ denote the coefficient in Lemma~\ref{lem:digitsums}, computed from that side's row, and let $c_{S,x}$ be its corresponding row constant. For each triple of row groups $\boldsymbol g=(g_0,g_1,g_2)$, define $C_S(\boldsymbol g)(x+3y)=c_{S,x}-2g_{x,y}$. Then
\begin{align*}
 N_A(\kappa)&=\sum_{\boldsymbol g:\,\kappa(C_A(\boldsymbol g))=\kappa}
                 \prod_{x=0}^2 N_{A,x}(g_x),\\
 N_B(\kappa)&=\sum_{\boldsymbol g:\,\kappa(C_B(\boldsymbol g))=T-\kappa}
                 \prod_{x=0}^2 N_{B,x}(g_x).
\end{align*}
The sums include all nonempty digit groups. Every D-optimal lift pair lies in a matched class, and the number of pairs satisfying the order-nine key condition is exactly $\sum_\kappa N_A(\kappa)N_B(\kappa)$.
\end{proposition}
\begin{proof}
The row choices are independent, so their counts multiply. Distinct triples of row groups partition all full lifts. Their length-nine compressions, and hence their keys, are determined by Lemma~\ref{lem:digitsums}. Summing over the appropriate keys gives the stated counts, and Proposition~\ref{prop:key} gives the necessary matching condition.
\end{proof}

Nothing in this counting step requires eleven columns. The next statement
shows how the procedure transfers to another odd abelian factor and bounds
the amount of class-counting work.

\begin{proposition}\label{prop:general9}
Let $H$ be a finite abelian group of odd order $w>1$.  On
$\Z_3\times H$, let $U_a,U_b$ take values in $\{\pm1,\pm3\}$, have
combined autocorrelation six off zero, and satisfy
$(\sum U_a)^2+(\sum U_b)^2=36w-2$.
The digit parametrisation, exact class-count formulas and order-nine
key extend to lifts on $\Z_9\times H$, with
\[
 T=(18w-2,0,0),\qquad
 \nu_1=\frac{9w-1}{2},\qquad \nu_3=\frac{3w+1}{2}.
\]
There are $3^{(9w-1)/2}$ lift pairs before matching.  Nevertheless, the
exact class counts can be aggregated from at most $(w+1)^6$ triples of
row digit groups on each side, without enumerating these full lifts.
Every D-optimal lift pair belongs to a complementary key class.
\end{proposition}

\begin{proof}
Corollary~\ref{thm:subgroup_three} gives the fibre counts and the number
of lift pairs.  In each row, sum over $H$ instead of $\Z_{11}$.
Lemma~\ref{lem:digitsums} still gives
$A_9(x+3y)=c_x-2g_{x,y}$ and the same generating-function coefficients.
The required PSD sum is now $18w-2$, so Proposition~\ref{prop:key}
gives the stated target.  The independent row choices give exactly
the class-count formulas of Proposition~\ref{prop:classcounts} with
this target.

For a row with $n=p_x+m_x\le w$ nonhomogeneous entries, every signed
digit sum lies in $[-m_x,p_x]$, an interval of $n+1$ integers, and
$g_{x,0}+g_{x,1}+g_{x,2}=p_x-m_x$.
Thus the first two coordinates determine the third, leaving at most
$(n+1)^2\le(w+1)^2$ nonempty groups.  Combining three rows gives
the bound $(w+1)^6$.
\end{proof}

The row coefficients themselves can be computed by successively
multiplying the three-term factors in Lemma~\ref{lem:digitsums} and
collecting equal monomials after each multiplication.
Thus this preprocessing uses polynomially many arithmetic operations
in $w$; it does not make enumeration of the full lifts polynomial.
The orbit reduction also extends: translations $(j,z)\mapsto(j+3,z)$
on either sequence and the common map $(j,z)\mapsto(4j,z)$ generate
$\Z_3^3$.  Each side has at least $(3w-1)/2>w$ nonhomogeneous fibres,
so its digits occupy two rows; the stabiliser argument of
Proposition~\ref{prop:sym} again makes this action free.
For $h\in H\setminus\{0\}$ the additive key is
\[
 \PAF_a(3,h)=\sum_{x=0}^2\sum_{y\in\Z_3}\sum_{z\in H}
 a(x+3y,z)a(x+3y+3,z+h).
\]
It separates by rows exactly as in Lemma~\ref{lem:additive}; matching
these necessary conditions is followed by the full autocorrelation test.

The counts require only row-group data, before any full-length lift is enumerated. They refer to the unreduced lists of lifts; orbit reductions are treated separately below. In the notation of Section~\ref{sec:search}, these are the counts with $\alpha=1$.

\subsection{Additive matching and orbit reduction}
The class key leaves further correlation equations to impose. Their rowwise
additivity provides an integer key for matching separately generated lifts.

\begin{lemma}\label{lem:additive}
For $t\in\Z_{11}$ let $k_t\in\Z_{99}$ be the element with $k_t\equiv3\pmod 9$ and $k_t\equiv t\pmod{11}$. Then $\PAF_a(k_t)=G_0(t)+G_1(t)+G_2(t)$, where
\[
G_x(t)=\sum_{y\in\Z_3}\sum_{z\in\Z_{11}}a(x+3y,z)\,a(x+3y+3,z+t)
\]
depends only on row $x$ of $U_a$ and its digits.
\end{lemma}

\begin{proof}
The elements $j$ and $j+3$ of $\Z_9$ lie in the same coset of $\{0,3,6\}$.
\end{proof}

The ten conditions $\PAF_a(k_t)+\PAF_b(k_t)=2$, $t\ne0$, are thus exact integer conditions with one term per row on each side, and serve as the meet-in-the-middle keys. They are necessary conditions; every match is subsequently tested at all shifts.

Some lifts differ only by transformations that preserve the compressed pair.
The next result identifies such transformations and shows how to choose
representatives without losing any solution up to this action. The generated
group has order $27$, describing pair orbits rather than the work saved on
the two separate lists.

\begin{proposition}\label{prop:sym}
Let $(U_a,U_b)$ satisfy Theorem~\ref{thm:twolevel}\,(i), with sums $13$ and $15$. Let $\tau$ be the translation of $\Z_{99}$ by $33$ and $\mu$ the automorphism $i\mapsto67i$, acting on functions by $(\tau a)(i)=a(i-33)$ and $(\mu a)(67i)=a(i)$.
\begin{enumerate}
\item[\textup{(a)}] $\tau$ fixes $U_a$, adds $2$ to every digit and preserves the key; it acts freely on the lifts of $U_a$ when $\nu_1(U_a)>0$.
\item[\textup{(b)}] $\mu$ fixes $U_a$, adds $x$ to every digit in row $x$, and maps the key by $f(\kappa_0,\kappa_1,\kappa_2)=(\kappa_0,\,\kappa_2-\kappa_1,\,-\kappa_1)$. The map $f$ has order $3$, and its fixed points are the keys with $\kappa_1=\kappa_2=0$.
\item[\textup{(c)}] The group $\Gamma\cong\Z_3^3$ generated by $\tau$ on $a$, $\tau$ on $b$, and $\mu$ on both preserves D-optimality and $(U_a,U_b)$, maps the lift class of $\kappa$ onto that of $\kappa$, $f(\kappa)$ or $f^2(\kappa)$, keeps $N_A$ and $N_B$, and acts freely on lift pairs.
\end{enumerate}
Hence every D-optimal lift of $(U_a,U_b)$ is found, up to $\Gamma$, by solving one class in each orbit of $f$, with one lift from each $\tau$-orbit on each side.
\end{proposition}

\begin{proof}
Since $33\equiv6\pmod9$ and $33\equiv0\pmod{11}$, $\tau$ maps $(x+3y,z)$ to $(x+3(y+2),z)$. Since $67\equiv4\pmod9$, $67\equiv1\pmod{11}$ and $4(x+3y)\equiv x+3(x+y)\pmod 9$, $\mu$ maps $(x+3y,z)$ to $(x+3(y+x),z)$. Both maps preserve every coset $\{x,x+3,x+6\}\times\{z\}$, hence $U_a$. For the key, $\PAF_{\mu a}(67k)=\PAF_a(k)$, so the 9-compression $A'_9$ of $\mu a$ has $\PAF_{A'_9}(k)=P_{7k}$, as $4^{-1}\equiv7\pmod 9$; this gives $(P'_1,P'_2,P'_3,P'_4)=(P_2,P_4,P_3,P_1)$ and the map $f$. Translations of either block and a common automorphism preserve D-optimality, and the remaining transformation statements are direct checks. Each compression has at least $33-17=16$ nonhomogeneous fibres, so its digits occur in at least two rows. A stabilising element of $\Gamma$ adds $2t_a+mx$ to the digits of $a$ and $2t_b+mx$ to those of $b$. Comparing two rows forces $m=0$, then $t_a=t_b=0$. Thus every pair orbit has $27$ elements.
\end{proof}

\subsection{Completeness of lifting}\label{subsec:search_scope}

The preceding results give a complete search for the lifts of a fixed
compressed pair. The next statement records what must be enumerated and
which reductions preserve completeness. It does not prescribe which
compressed pairs a discovery search should select.

\begin{proposition}\label{prop:lift_completeness}
Let $(U_a,U_b)$ have sums $13$ and $15$ and satisfy
Theorem~\ref{thm:twolevel}\,\textup{(i)}.  Form all nonempty digit groups and
all matched lift classes using their full counts.  In each class, enumerate
the lifts on both sides and match the additive autocorrelation keys of
Lemma~\ref{lem:additive}, retaining every match for the final test
$\PAF_a(k)+\PAF_b(k)=2$ at every nonzero $k\in\Z_{99}$.
If pruning discards candidates only through necessary conditions, the output
consists precisely of the D-optimal lifts of $(U_a,U_b)$.
If instead one uses a representative of each class orbit under $f$ and one
representative of each $\tau$-orbit on either side, the output contains a
representative of every $\Gamma$-orbit of D-optimal lifts and no other pairs.
\end{proposition}

\begin{proof}
Every lift has a unique digit assignment and hence belongs to a row digit
group on each row.  Lemma~\ref{lem:digitsums} counts all such assignments,
and Proposition~\ref{prop:key} places every D-optimal lift pair in a matched
class.  Its two additive keys satisfy the matching conditions, so complete
enumeration and matching retain it.  Necessary pruning conditions cannot
remove it, and the final autocorrelation test is equivalent to D-optimality.
Conversely, every output passes that test.  Proposition~\ref{prop:sym} shows
that the stated orbit reductions preserve the compressed pair and
D-optimality and retain a representative of each $\Gamma$-orbit.
\end{proof}

Pruning means rejecting a candidate as soon as a necessary condition fails.
For a D-optimal lift at length $99$, each individual PSD at a nontrivial
character is at most $196$. Thus, if $z$ is a partial Fourier sum at such a character and $R$
bounds the absolute contribution of the unassigned terms, then
$|z|>14+R$ excludes that lift. An implementation must use exact bounds or
conservative error bounds that retain uncertain candidates. Completeness
within a fixed compression does not imply that every compression or class
was visited by the reported search.

\subsection{Search organisation}\label{sec:search}

For selecting work, a class with full side counts $N_A$ and $N_B$ contains
$N_AN_B$ candidate pairs.  When one side fits in the store, $N_A+N_B$ is an
idealised meet-in-the-middle work estimate, assuming comparable generation
and processing costs for each lift.  This motivates
\[
 E(\kappa)=\frac{N_A(\kappa)N_B(\kappa)}
 {N_A(\kappa)+N_B(\kappa)}.
\]
If storage requires chunks, the opposite list may have to be generated or
probed repeatedly.  Processing matching-key collisions and performing the final tests add work; spectral pruning and the sizes and organisation of the
row groups also affect the actual work.  The ranking by $E$ is therefore a
search heuristic, not a proved optimal ranking or an exact runtime model.

Preliminary scoring truncates the counts using an integer cutoff $\alpha\ge1$. For $S\in\{A,B\}$ let
$N_{S,\alpha}(\kappa)$ be the sum of the product multiplicities over
triples of row digit groups in that class for which
$\prod_x N_{S,x}(g_x)\ge\alpha$, and put
\[
 E_\alpha(\kappa)=
 \frac{N_{A,\alpha}(\kappa)N_{B,\alpha}(\kappa)}
 {N_{A,\alpha}(\kappa)+N_{B,\alpha}(\kappa)},
\]
with value zero when both counts vanish.  The default $\alpha=10^7$ gives a
\emph{truncated class score}.  Truncation precedes aggregation by key: a
discarded combination of row groups can contribute to an otherwise large class, so this
score is not the exact class efficiency.  With $\alpha=1$ all nonempty
groups contribute and $N_{S,1}=N_S$, $E_1=E$.  Exact class counts, including
the counts reported for the explicit construction, use $\alpha=1$.

The targeted search retains compressed pairs whose largest truncated class
score is at least $\theta=3.5\cdot10^8$ and ranks them by that score.  It
recomputes the full counts for selected pairs and lifts classes with
$E\ge\theta$.  Both preliminary
selection and the class threshold restrict the search.  They provide no
global coverage statement; a compressed pair discarded at the preliminary
stage may still contain a class with a large exact class score.

The length-$99$ search follows the compressions in Lemma~\ref{lem:split}. A \emph{compressed library} stores representatives of equivalence classes of compressed pairs. The following stages use the scores above to schedule the remaining work.

\paragraph{Stage 1: Compressed libraries.}
Form integer sequences of lengths $3$ and $11$ with the sums $13$ or $15$, odd entries bounded by $99/d$, and the energy and spectral bounds from \eqref{eq:const}. Match their autocorrelation vectors to obtain complementary compressed pairs, reducing by independent translations and reversals and a common multiplier. The length-nine conditions will be imposed by the integer key, so a separate length-nine library is unnecessary for the lift.

\paragraph{Stage 2: Compressed arrays.}
For each length-$11$ pair, construct $3\times11$ arrays with those column sums, entries in $\{\pm1,\pm3\}$, and a row-sum pair belonging to a stored length-three class. Expand the library representatives under their equivalence operations as needed for the fixed column sums; the row sums need not equal the stored representatives themselves. Retain pairs with compressed autocorrelation constants $(202,6)$. Reject partial assignments that exceed the homogeneous-fibre budget or a necessary spectral bound. Choose one representative from each orbit under the group of order $18$ generated by independent row rotations and common row reversal. Score the retained compressed pairs using their row digit groups.

\paragraph{Stage 3: Class matching.}
Compute the full class counts by Proposition~\ref{prop:classcounts}. In a selected class, enumerate its row digit assignments, apply necessary spectral bounds, and match the ten additive keys of Lemma~\ref{lem:additive}. Use the orbit reduction of Proposition~\ref{prop:sym}, and test every surviving pair by its complete integer autocorrelation vector.

\paragraph{Stage 4: The driver.}
Process library classes in batches and store the resulting compressed pairs in a pool ordered by their largest truncated class score. The discovery run used batches of $64$ library classes. Each lifting round selected $640$ compressed pairs from the pool for processing on $32$ workers. The driver stopped when a pair passed the final exact test. The counts used to record completed work are specified below.

For completed class orbits define
\[
 W=\sum_{\kappa}N_A(\kappa)N_B(\kappa),
\]
where the sum uses full counts and includes each class and its distinct
$f$-images once for each compressed pair.  Here completion means that the
required enumeration and matching within the selected class orbit have
finished.  The value $W$ records the candidate-pair volume represented by
those orbits before spectral pruning and symmetry reduction.  It is neither
the number of pairs individually subjected to every final test nor a
probability of finding a design.

\section{An invariant lift at length 85}\label{sec:lift85}

At length $85$, the compression-and-matching approach has a smaller invariant
branch. Under $\Z_{85}\cong\Z_5\times\Z_{17}$, multiplication by $69$ acts as
$(j,z)\mapsto(-j,z)$. An invariant binary sequence therefore has fibres
\[
 (a(0,z),\ldots,a(4,z))=(a_0,a_1,a_2,a_2,a_1),\qquad
 U_a(z)=a_0+2a_1+2a_2.
\]
For $\sigma\in\{\pm1\}$, the triples $(a_0,a_1,a_2)$ allowed by a compressed
entry are
\[
\begin{array}{c|c}
 U_a(z)&(a_0,a_1,a_2)\\\hline
 5\sigma&(\sigma,\sigma,\sigma)\\
 3\sigma&(-\sigma,\sigma,\sigma)\\
 \sigma&(\sigma,\sigma,-\sigma),\ (\sigma,-\sigma,\sigma).
\end{array}
\]
Thus only entries of absolute value one require a choice. If $\nu_j$ counts
entries of absolute value $j$ in both compressed sequences, then a pair
complementary with constants $(178,10)$ has
\[
 \nu_1+\nu_3+\nu_5=34,\qquad
 \nu_1+9\nu_3+25\nu_5=178,
 \qquad \nu_1=16+2\nu_5\le28.
\]
Here the inequality follows from $\nu_3+3\nu_5=18$. Consequently, this
compressed pair has exactly $2^{\nu_1}\le2^{28}$ invariant lift pairs before
the remaining channel conditions are imposed.

The nontrivial channels also admit exact integer matching. For a lift $a$, put
\[
 P_a=(a_0-a_2)/2,\qquad Q_a=(a_1-a_2)/2,
\]
as functions on $\Z_{17}$, and define
\begin{align*}
 K_{a,0}(k)&=\PAF_{P_a}(k)+\PAF_{Q_a}(k),\\
 K_{a,1}(k)&=\sum_z\bigl(P_a(z)Q_a(z+k)+Q_a(z)P_a(z+k)\bigr)
                    -\PAF_{Q_a}(k).
\end{align*}
Both are integer-valued and even in $k$, so their values at $0\le k\le8$
form an eighteen-coordinate key. Reflection identifies opposite channels.
The next result shows that matching these integer vectors enforces all
remaining channel conditions within the invariant family.

\begin{proposition}
\label{prop:lift85}
Let $(U_a,U_b)$ have entries in $\{\pm1,\pm3,\pm5\}$, sums $7$ and $17$,
and complementary constants $(178,10)$. Two invariant binary lifts form a
D-optimal pair on $\Z_{85}$ if and only if
\[
 K_{a,0}+K_{b,0}=42\delta_0,\qquad K_{a,1}+K_{b,1}=0.
\]
\end{proposition}

\begin{proof}
Put $\gamma=2\cos(2\pi/5)$, so $\gamma^2=1-\gamma$.
The channel at a primitive fifth root is
\[
 a_0+\gamma a_1+(-1-\gamma)a_2=2(P_a+\gamma Q_a),
\]
whose autocorrelation is $4(K_{a,0}+\gamma K_{a,1})$.
The other distinct nontrivial channel is its conjugate under
$\gamma\mapsto-1-\gamma$; reflection invariance identifies the channels
with opposite indices. Since $1,\gamma$ are linearly independent over
$\Q$, the stated integer identities are equivalent to the combined
autocorrelation $168\delta_0$ in both nontrivial channels. Their spectra
therefore sum to $168$ at every character of $\Z_{17}$. The compression
channel already has spectral sum $168$ at nontrivial characters and
$7^2+17^2=338$ at the trivial character. Together these are precisely the
D-optimal spectral conditions at length $85$.
\end{proof}

The invariant construction enumerates the allowed fibre patterns and matches the two lists by the complementary integer keys of Proposition~\ref{prop:lift85}. Each match is checked through $\lambda_X(k)+\lambda_Y(k)=31$ for every nonzero shift. The local patterns also apply on $\Z_5\times H$ for other odd abelian groups $H$: Theorem~\ref{thm:channels} supplies the corresponding spectral target. Thus invariant five-entry fibres and unrestricted three-entry fibres are two ways to implement the general lifting method.

\section{Applications to two unresolved parameter sets}\label{sec:designs}

The two lifting branches yield supplementary difference sets with parameters
$(99;43,42;36)$ and $(85;39,34;31)$. The first gives a D-optimal matrix of order
$198$; the second realises a new parameter set at the already known order $170$.
The explicit blocks below establish existence independently of the search that
produced them.

\subsection{A D-optimal matrix of order 198}
The method of Sections~\ref{sec:lift99} and~\ref{sec:search} gives the following
pair, with no fixing multiplier imposed on its blocks.

\begin{theorem}\label{thm:main}
Let $X$ and $Y$ be the following subsets of $\Z_{99}$:
\begin{align*}
X&=\{0,1,3,5,11,16,22,23,24,28,29,31,36,37,40,44,45,46,48,50,54,55,56,58,59,60,\\
&\qquad 61,63,72,73,74,75,79,81,82,84,86,88,89,90,92,93,96\},\\
Y&=\{0,2,3,11,12,19,21,22,24,27,31,34,35,38,43,44,49,50,51,54,55,57,59,60,69,71,\\
&\qquad 72,74,75,77,78,79,81,84,85,88,89,90,92,94,96,97\}.
\end{align*}
The pair $(X,Y)$ forms a supplementary difference set $(99;43,42;36)$ in $\Z_{99}$. A D-optimal design of order $198$ exists, and the maximal determinant of a $\pm1$ matrix of order $198$ is exactly $2^{99}\cdot197\cdot98^{98}$.
\end{theorem}

\begin{proof}
Direct counting in the displayed blocks gives $|X|=43$, $|Y|=42$, and
\[
\#\{(x,x')\in X^2:x-x'=t\}
+\#\{(y,y')\in Y^2:y-y'=t\}=36
\qquad(1\le t\le98),
\]
where differences are taken modulo $99$. Thus $(X,Y)$ has the stated
parameters. The associated sign sequences have autocorrelation sum $198$ at
zero and $2$ elsewhere. By Lemma~\ref{lem:twocirc}, their two-circulant matrix
satisfies
\[
MM^{\top}=I_2\otimes(196I_{99}+2J_{99}).
\]
Each diagonal block has eigenvalues $394$ once and $196$ with multiplicity
$98$. Therefore $|\det M|=394\cdot196^{98}=2^{99}\cdot197\cdot98^{98}$,
which attains the Ehlich--Wojtas bound.
\end{proof}

Table~\ref{tab:U} records the compression and digits of this solution, making
its position in the lifting construction explicit. At a fibre $(x,z)$, an
entry $\pm3$ assigns its sign to all three positions. An entry $\sigma=\pm1$
with digit $\rho$ assigns $-\sigma$ to the position satisfying
$i\equiv x+3\rho\pmod9$, $i\equiv z\pmod{11}$, and $\sigma$ to the other two.

\begin{table}[H]
\centering\small
\caption{The compressed pair of the design, rows $x\in\Z_3$ and columns $z\in\Z_{11}$, with the digits $\rho(x,z)$ as subscripts.}\label{tab:U}
\setlength{\tabcolsep}{4.2pt}
\begin{tabular}{c*{11}{c}}
\toprule
$U_a$ & 0 & 1 & 2 & 3 & 4 & 5 & 6 & 7 & 8 & 9 & 10\\
\midrule
0 & $+1_{0}$ & $+1_{0}$ & $-1_{1}$ & $-1_{2}$ & $-1_{2}$ & $-1_{0}$ & $+1_{0}$ & $+1_{1}$ & $-1_{1}$ & $+1_{1}$ & $+1_{0}$ \\
1 & $-3$ & $+1_{0}$ & $-1_{1}$ & $+1_{1}$ & $+1_{0}$ & $-1_{1}$ & $-1_{1}$ & $-1_{2}$ & $+3$ & $+1_{1}$ & $+3$ \\
2 & $-1_{1}$ & $-3$ & $+3$ & $+3$ & $-1_{2}$ & $+1_{1}$ & $+1_{1}$ & $+1_{0}$ & $+1_{0}$ & $+1_{1}$ & $+3$ \\
\midrule
$U_b$ & 0 & 1 & 2 & 3 & 4 & 5 & 6 & 7 & 8 & 9 & 10\\
\midrule
0 & $+1_{0}$ & $-1_{0}$ & $-3$ & $-1_{0}$ & $+1_{0}$ & $-1_{1}$ & $+1_{0}$ & $-1_{0}$ & $+1_{2}$ & $+1_{1}$ & $-1_{2}$ \\
1 & $-3$ & $+1_{2}$ & $+1_{2}$ & $+3$ & $+3$ & $+1_{1}$ & $+1_{1}$ & $+3$ & $-1_{2}$ & $-1_{0}$ & $+1_{2}$ \\
2 & $-3$ & $+1_{2}$ & $-1_{1}$ & $+3$ & $-1_{2}$ & $-1_{1}$ & $+1_{1}$ & $+3$ & $+1_{0}$ & $+3$ & $+3$ \\
\bottomrule
\end{tabular}
\end{table}

The two arrays have $26$ and $23$ nonhomogeneous fibres, respectively, so their
unrestricted digit assignments give $3^{26}3^{23}=3^{49}$ lift pairs. Their
9-compressions are
\[
A_9=(-5,-3,1,1,3,1,5,3,7),\qquad
B_9=(-1,7,1,-1,1,5,-1,1,3).
\]
Their keys are $(148,24,12)$ and $(48,-24,-12)$, whose sum is $(196,0,0)$ as
required by Proposition~\ref{prop:key}.

The following calculation measures the exact reduction for the displayed
compression and identifies the classes selected by the discovery threshold.
It does not describe every compression encountered in the search.

\begin{proposition}\label{prop:structure}
Using every digit group, equivalently $\alpha=1$, the pair in
Table~\ref{tab:U} has $149$ nonempty matched lift classes. Together these contain
exactly $26{,}710{,}293{,}936{,}839{,}797{,}449$ pairs, approximately $1.12\cdot10^{-4}$ of the $3^{49}$ lift pairs. Exactly three
classes have $E\ge3.5\cdot10^8$, with side-$A$ keys
\[
(148,24,12),\qquad(148,-12,-24),\qquad(148,-12,12).
\]
They form one orbit of the key action in Proposition~\ref{prop:sym}. Each has
$N_A=643{,}674{,}780$, $N_B=868{,}628{,}835$ and
$E\simeq3.697\cdot10^8$. The displayed solution belongs to the first class.
\end{proposition}

\begin{proof}
For each row of the two arrays, the coefficient formula in
Lemma~\ref{lem:digitsums} gives the multiplicity of every signed digit-sum
triple. Forming the three-row products and aggregating them by the integer key
of Proposition~\ref{prop:key}, with no discarded groups, gives the stated
counts. The fraction is $3^{-49}\sum_{\kappa}N_A(\kappa)N_B(\kappa)$, with
the sum over all matched keys. Applying the map $f$ of
Proposition~\ref{prop:sym} gives the displayed three-key orbit.
\end{proof}

The exact key therefore excludes $99.988838\%$ of the unrestricted lift pairs for this compression, a factor of approximately $8{,}959$, without discarding any D-optimal lift. The threshold then selects three classes with a combined candidate volume of $1.6773\cdot10^{18}$ pairs, or $6.28\%$ of the volume satisfying the key condition. This further selection can omit solutions. These reductions refer to the displayed compression, not to every compressed pair encountered in the search.

Within each of the three selected classes, direct pairwise testing would consider $N_AN_B=559{,}114{,}474{,}270{,}281{,}300$ pairs. Separate-list matching instead starts from $N_A+N_B=1{,}512{,}303{,}615$ lifts before pruning and orbit reduction. The ratio is about $3.70\cdot10^8$, but it compares candidate pairs with an idealised generation and lookup count, not measured execution times. Collision processing, final tests and any repeated probing caused by storage limits also contribute to runtime.

These are full class counts; the preliminary truncated score is used to
select and schedule the targeted search. The reported discovery run processed $512$ of
the $51{,}155$ length-$11$ library classes on a 32-core workstation. It reached
the solution after $1{,}256{,}320$ compressed pairs had entered the lifting
stage, using $2{,}993$ CPU-hours for that stage and approximately $3{,}200$
CPU-hours overall. These are the reported discovery costs. No controlled timing comparison with an unstructured exhaustive lift is reported, so the candidate reductions above are not asserted as observed speed-up factors. The run searched selected classes and
does not enumerate all order-$198$ designs.

\subsection{The parameter set \texorpdfstring{$(85;39,34;31)$}{(85;39,34;31)}}
The multiplier-restricted lifting branch in Section~\ref{sec:lift85} gives the
following pair. D-optimal matrices of order $170$ were already known
\cite{GK99}; the contribution here is the previously unresolved supplementary
difference set parameter set.

\begin{theorem}\label{thm:85}
Let $X$ and $Y$ be the following subsets of $\Z_{85}$:
\begin{align*}
X&=\{1, 4, 7, 8, 9, 15, 17, 20, 21, 22, 25, 26, 28, 29, 30, 32, 33, 34, 35, 38,\\
&\qquad 40, 42, 43, 44, 46, 48, 51, 58, 61, 62, 67, 68, 69, 70, 72, 73, 77, 82, 83\},\\
Y&=\{0, 1, 2, 4, 6, 10, 16, 17, 19, 20, 21, 28, 29, 32, 34, 36, 39, 46, 48, 51,\\
&\qquad 53, 56, 59, 60, 62, 68, 69, 74, 75, 76, 80, 82, 83, 84\}.
\end{align*}
The pair $(X,Y)$ forms a supplementary difference set $(85;39,34;31)$ in $\Z_{85}$. Both blocks are invariant under multiplication by 69, and yield a $170 \times 170$ D-optimal matrix.
\end{theorem}

\begin{proof}
The displayed blocks have sizes $39$ and $34$, and their combined difference
multiplicity is $31$ at each of the $84$ nonzero residues modulo $85$.
Multiplication of their elements by $69$ modulo $85$ leaves each block
unchanged. Their sign sequences therefore have autocorrelation sum $2$ away
from zero, and Lemma~\ref{lem:twocirc} gives a D-optimal matrix of order $170$.
\end{proof}

For this pair the compressions onto $\Z_{17}$ are
\begin{align*}
 U_a&=(-3,-1,3,3,-3,1,3,1,-1,-3,1,1,1,3,1,-1,1),\\
 U_b&=(-5,1,-3,3,1,1,1,3,1,3,3,1,-1,5,1,1,1).
\end{align*}
They have combined fibre counts $\nu_1=20$, $\nu_3=12$, $\nu_5=2$. A five-entry fibre with sum $\pm1$ has ten unrestricted binary patterns, one with sum $\pm3$ has five, and one with sum $\pm5$ has one. The unrestricted compressed pair therefore has $10^{20}5^{12}\simeq2.4414\cdot10^{28}$ binary lift pairs, whereas multiplier-$69$ invariance leaves exactly $2^{20}=1{,}048{,}576$. This saving comes from restricting the construction to invariant blocks and need not preserve every unrestricted solution. Integer matching and final difference tests are still required. The displayed lift satisfies the target $(42\delta_0,0)$ of Proposition~\ref{prop:lift85} and the direct difference test above.

The recorded collection contains $56$ solutions. Reversal of one block is
allowed in SDS equivalence \cite[Section~2]{Do11}, but may transpose the
associated matrix rather than preserve its matrix class. The following
classification explains why seven SDS representatives give fourteen
matrix classes. It concerns only this recorded collection.

\begin{proposition}\label{prop:matrix85}
The $56$ recorded solutions represent exactly seven SDS equivalence classes,
with representatives labelled $1,3,5,7,9,11,13$ in Appendix~\ref{app:85}.
Their associated matrices represent exactly fourteen classes under signed
row and column permutations, and exactly seven classes when transposition is
also allowed. The latter classes have the same seven representative labels.
\end{proposition}

\begin{proof}
Canonicalising the $56$ pairs under independent translations and common units
modulo $85$ gives the fourteen appendix entries. Write $(X_i,Y_i)$ for entry
$i$ and $M_i$ for its two-circulant matrix. Table~\ref{tab:matrix85} gives
\[
X_{2k}=X_{2k-1},\qquad Y_{2k}=c_k-Y_{2k-1},\qquad
(c_1,\ldots,c_7)=(32,4,3,3,3,3,3),
\]
where $c-Y=\{c-y\bmod85:y\in Y\}$. Each successive pair of entries is
therefore equivalent under reversal and translation of its second block,
giving at most seven SDS classes.

To separate the matrices, define the four-column invariant
\[
h_2(M)=\#\left\{\{i,j,k,l\}:i<j<k<l,\quad
\left|\sum_rM_{ri}M_{rj}M_{rk}M_{rl}\right|=2\right\}.
\]
Row signs cancel, column signs disappear under the absolute value, and
permutations reorder the terms or the four-subsets. Hence $h_2$ is invariant
under signed row and column permutations. Enumerating all
$\binom{170}{4}=33{,}585{,}370$ four-subsets gives the values in
Table~\ref{tab:matrix85}, distinct for all fourteen entries.

Independent block translations and a common multiplier induce row and column
permutations of the two-circulant matrix. If $R$ reverses cyclic indices and
$D=\operatorname{diag}(R,-R)$, direct multiplication gives
\[
M(A,B^{\top})=D\,M(A,B)^{\top}D.
\]
Thus the matrices of each successive pair are equivalent after transposition,
and every SDS equivalence operation induces signed row and column
equivalence, possibly followed by transposition. The seven unordered pairs
$\{h_2(M),h_2(M^{\top})\}$ in the table are distinct, separating the seven
matrix classes even with transposition and hence the seven SDS classes.
\end{proof}

\begin{table}[H]\centering\small
\caption{The seven SDS classes among the recorded order-85 solutions and the fourteen associated matrix classes. Labels refer to Appendix~\ref{app:85}; in each row $X_j=X_i$ and the displayed relation for $Y_j$ is modulo $85$. The matrices $M_i$ and $M_j$ are equivalent after transposition, with $h_2(M_j)=h_2(M_i^\top)$.}\label{tab:matrix85}
\begin{tabular}{rrlrr}\toprule
SDS rep. $i$ & Partner $j$ & $Y_j$ & $h_2(M_i)$ & $h_2(M_j)$\\\midrule
1 & 2 & $32-Y_1$ & 4,096,830 & 4,054,160\\
3 & 4 & $4-Y_3$ & 4,021,010 & 4,061,640\\
5 & 6 & $3-Y_5$ & 4,077,705 & 4,010,215\\
7 & 8 & $3-Y_7$ & 4,039,285 & 4,108,645\\
9 & 10 & $3-Y_9$ & 4,101,505 & 4,054,840\\
11 & 12 & $3-Y_{11}$ & 4,110,515 & 4,083,740\\
13 & 14 & $3-Y_{13}$ & 4,027,810 & 4,062,235\\\bottomrule
\end{tabular}\end{table}

All fourteen appendix entries are retained because they represent different
matrix classes under signed row and column permutations. Their pairing in
Table~\ref{tab:matrix85} records precisely why only seven are needed as SDS
representatives.

\section{Conclusion and further applications}\label{sec:open}
The procedure developed here turns complementary compressed pairs into explicit local choices that can be counted and matched. Applying it resolves the SDS parameter sets $(99;43,42;36)$ and $(85;39,34;31)$, giving a D-optimal matrix of the previously unresolved order $198$ and new parameters at order $170$. The recorded length-$85$ collection gives seven SDS classes and fourteen matrix classes under signed row and column permutations. The retained representatives show exactly how allowing transposition joins the matrix classes in pairs.

The main practical benefit is the information available before full lifts are generated. On $\Z_9\times H$, signed digit sums give exact class sizes using polynomially many row-group combinations. Researchers can use these counts to identify compatible classes and assess their storage requirements before enumeration. The three-integer spectral key excludes incompatible pairs; additive autocorrelation matching then avoids comparing every pair of retained lifts. For the displayed length-$99$ compression, the exact key reduces the candidate volume by a factor of about $8{,}959$. For the length-$85$ example, reflection invariance leaves $2^{20}$ local lift pairs. The key preserves all D-optimal lifts of its fixed compression, whereas reflection invariance and the discovery thresholds restrict the searched family. Neither reduction is a measured runtime factor.

The subgroup criterion applies to other odd abelian groups, and the homogeneous-fibre bound can reject impossible fixing multipliers before enumeration. Each application still needs suitable compressed pairs and explicit local lift patterns. The remaining cost depends on class sizes, pruning and storage. Prime lengths have no nontrivial proper compression subgroup; larger fibres may require larger integer keys. The polynomial bound covers class counting, not the full construction problem.

Length $153=9\cdot17$ is a direct further application of the order-nine branch. Compression along its order-three subgroup gives arrays on $\Z_3\times\Z_{17}$ with the same three minority-sign choices. The signed digit sums determine a length-nine compression and an integer-key target $(304,0,0)$. The two parameter sets recorded in \cite{Do22} are $(153;70,66;60)$ and $(153;72,65;61)$, with sequence sums $(13,21)$ and $(9,23)$. Each admissible compressed pair leaves $3^{76}$ unrestricted lift pairs. Proposition~\ref{prop:general9} transfers the exact class calculation, but feasible class sizes and lifting costs still need to be established computationally. This is a further case suited to the method, not the smallest remaining open matrix order; no construction of order $306$ is asserted here.

Appendix~\ref{app:reflection} applies the framework to reflection-invariant sequences on $\Z_3^2\times\Z_{11}$. Its exact criterion gives additional constraints for a search on this noncyclic group, without establishing nonexistence. Further computations are needed to decide whether that group supports the same SDS parameters and to determine how many inequivalent order-$198$ matrices the method produces.

\paragraph{Data availability.}
The explicit blocks and the fourteen order-$170$ matrix representatives are included in the article. The recorded length-$85$ solutions, equivalence mappings and four-column profile data, together with verification routines, are available from the corresponding author on request. The numerical class counts in Proposition~\ref{prop:structure} are determined by the displayed compressed arrays and the finite coefficient formula in Proposition~\ref{prop:classcounts}.

\clearpage
\appendix
\section{Representatives at \texorpdfstring{$v=85$}{v=85}}\label{app:85}
The following fourteen representatives of $(85;39,34;31)$ are inequivalent under independent translations and a common multiplier. Each block is translated to its lexicographically smallest form. Under SDS equivalence, which also allows independent reversal, there are exactly seven classes among the recorded solutions, represented by labels $1,3,5,7,9,11,13$.

We retain all fourteen entries because their $170\times170$ matrices are pairwise inequivalent under signed row and column permutations (Proposition~\ref{prop:matrix85}). Successive entries $(1,2),(3,4),\ldots,(13,14)$ have the same first block, and their second blocks satisfy the reversal--translation relations in Table~\ref{tab:matrix85}. Each such pair therefore represents one SDS class and two matrix classes exchanged by transposition.

\scriptsize
\begin{multicols}{2}
\noindent\textbf{[1]}
X = \{0, 1, 2, 3, 4, 5, 6, 7, 8, 10, 14, 18, 20, 25, 26, 29, 31, 33, 35, 36, 41, 44, 49, 50, 52, 53, 57, 58, 63, 64, 68, 70, 73, 75, 77, 78, 79, 80, 82\} \\
Y = \{0, 1, 2, 4, 5, 8, 12, 14, 17, 18, 20, 21, 23, 28, 30, 31, 32, 35, 42, 43, 44, 49, 50, 52, 53, 57, 61, 62, 63, 66, 68, 69, 77, 83\}

\vspace{4pt}\noindent\textbf{[2]}
X = \{0, 1, 2, 3, 4, 5, 6, 7, 8, 10, 14, 18, 20, 25, 26, 29, 31, 33, 35, 36, 41, 44, 49, 50, 52, 53, 57, 58, 63, 64, 68, 70, 73, 75, 77, 78, 79, 80, 82\} \\
Y = \{0, 1, 2, 4, 9, 11, 12, 14, 15, 18, 20, 24, 27, 28, 30, 31, 32, 34, 40, 48, 49, 51, 54, 55, 56, 60, 64, 65, 67, 68, 73, 74, 75, 82\}

\vspace{4pt}\noindent\textbf{[3]}
X = \{0, 1, 2, 3, 4, 5, 6, 7, 9, 10, 11, 13, 14, 20, 21, 25, 26, 30, 31, 33, 35, 37, 40, 44, 45, 47, 48, 53, 58, 59, 61, 65, 67, 69, 74, 76, 79, 80, 82\} \\
Y = \{0, 1, 2, 3, 4, 13, 15, 17, 21, 25, 28, 29, 30, 33, 37, 38, 39, 40, 44, 46, 47, 50, 53, 56, 61, 62, 65, 67, 68, 69, 70, 72, 75, 80\}

\vspace{4pt}\noindent\textbf{[4]}
X = \{0, 1, 2, 3, 4, 5, 6, 7, 9, 10, 11, 13, 14, 20, 21, 25, 26, 30, 31, 33, 35, 37, 40, 44, 45, 47, 48, 53, 58, 59, 61, 65, 67, 69, 74, 76, 79, 80, 82\} \\
Y = \{0, 1, 2, 3, 4, 9, 14, 17, 19, 20, 21, 22, 24, 27, 28, 33, 36, 39, 42, 43, 45, 49, 50, 51, 52, 56, 59, 60, 61, 64, 68, 72, 74, 76\}

\vspace{4pt}\noindent\textbf{[5]}
X = \{0, 1, 2, 3, 4, 5, 6, 10, 11, 14, 15, 16, 17, 20, 21, 22, 24, 27, 29, 32, 35, 37, 38, 40, 41, 44, 45, 48, 52, 54, 56, 57, 58, 60, 62, 63, 70, 77, 78\} \\
Y = \{0, 1, 2, 3, 5, 8, 9, 12, 14, 18, 21, 24, 25, 30, 32, 34, 38, 39, 40, 45, 47, 57, 59, 60, 62, 64, 65, 66, 71, 74, 75, 76, 79, 83\}

\vspace{4pt}\noindent\textbf{[6]}
X = \{0, 1, 2, 3, 4, 5, 6, 10, 11, 14, 15, 16, 17, 20, 21, 22, 24, 27, 29, 32, 35, 37, 38, 40, 41, 44, 45, 48, 52, 54, 56, 57, 58, 60, 62, 63, 70, 77, 78\} \\
Y = \{0, 1, 2, 3, 5, 9, 12, 13, 14, 17, 22, 23, 24, 26, 28, 29, 31, 41, 43, 48, 49, 50, 54, 56, 58, 63, 64, 67, 70, 74, 76, 79, 80, 83\}

\vspace{4pt}\noindent\textbf{[7]}
X = \{0, 1, 2, 3, 4, 5, 6, 8, 9, 12, 18, 19, 21, 25, 27, 28, 29, 33, 35, 40, 43, 44, 47, 48, 51, 53, 55, 56, 57, 62, 64, 72, 73, 74, 76, 77, 78, 79, 83\} \\
Y = \{0, 1, 2, 3, 6, 8, 10, 13, 15, 24, 25, 28, 33, 37, 41, 42, 43, 44, 47, 50, 53, 55, 61, 62, 64, 65, 67, 69, 70, 74, 75, 77, 80, 81\}

\vspace{4pt}\noindent\textbf{[8]}
X = \{0, 1, 2, 3, 4, 5, 6, 8, 9, 12, 18, 19, 21, 25, 27, 28, 29, 33, 35, 40, 43, 44, 47, 48, 51, 53, 55, 56, 57, 62, 64, 72, 73, 74, 76, 77, 78, 79, 83\} \\
Y = \{0, 1, 2, 3, 7, 8, 11, 13, 14, 18, 19, 21, 23, 24, 26, 27, 33, 35, 38, 41, 44, 45, 46, 47, 51, 55, 60, 63, 64, 73, 75, 78, 80, 82\}

\vspace{4pt}\noindent\textbf{[9]}
X = \{0, 1, 2, 3, 4, 5, 6, 7, 8, 10, 12, 14, 21, 23, 28, 30, 31, 33, 37, 42, 43, 45, 48, 49, 53, 57, 58, 62, 63, 64, 66, 67, 69, 74, 75, 77, 80, 81, 82\} \\
Y = \{0, 1, 2, 3, 5, 10, 11, 13, 14, 15, 17, 22, 27, 28, 30, 32, 36, 40, 41, 44, 47, 48, 50, 51, 54, 55, 57, 64, 66, 67, 70, 72, 78, 79\}

\vspace{4pt}\noindent\textbf{[10]}
X = \{0, 1, 2, 3, 4, 5, 6, 7, 8, 10, 12, 14, 21, 23, 28, 30, 31, 33, 37, 42, 43, 45, 48, 49, 53, 57, 58, 62, 63, 64, 66, 67, 69, 74, 75, 77, 80, 81, 82\} \\
Y = \{0, 1, 2, 3, 9, 10, 16, 18, 21, 22, 24, 31, 33, 34, 37, 38, 40, 41, 44, 47, 48, 52, 56, 58, 60, 61, 66, 71, 73, 74, 75, 77, 78, 83\}

\vspace{4pt}\noindent\textbf{[11]}
X = \{0, 1, 2, 3, 4, 5, 7, 8, 9, 10, 11, 15, 18, 20, 21, 23, 24, 26, 27, 29, 33, 35, 37, 40, 41, 42, 45, 46, 50, 52, 53, 55, 57, 58, 65, 67, 69, 78, 79\} \\
Y = \{0, 1, 2, 3, 6, 7, 9, 14, 18, 23, 24, 27, 30, 32, 34, 36, 37, 42, 46, 47, 52, 53, 56, 60, 61, 63, 64, 65, 67, 71, 72, 73, 75, 78\}

\vspace{4pt}\noindent\textbf{[12]}
X = \{0, 1, 2, 3, 4, 5, 7, 8, 9, 10, 11, 15, 18, 20, 21, 23, 24, 26, 27, 29, 33, 35, 37, 40, 41, 42, 45, 46, 50, 52, 53, 55, 57, 58, 65, 67, 69, 78, 79\} \\
Y = \{0, 1, 2, 3, 10, 13, 15, 16, 17, 21, 23, 24, 25, 27, 28, 32, 35, 36, 41, 42, 46, 51, 52, 54, 56, 58, 61, 64, 65, 70, 74, 79, 81, 82\}

\vspace{4pt}\noindent\textbf{[13]}
X = \{0, 1, 2, 3, 4, 5, 6, 8, 9, 10, 14, 16, 17, 18, 20, 22, 23, 27, 32, 33, 34, 37, 39, 41, 44, 45, 47, 52, 54, 55, 57, 58, 60, 62, 65, 67, 71, 76, 79\} \\
Y = \{0, 1, 2, 3, 6, 7, 9, 10, 15, 17, 26, 28, 29, 30, 37, 40, 44, 46, 49, 50, 53, 58, 59, 61, 62, 65, 66, 68, 69, 70, 74, 75, 76, 80\}

\vspace{4pt}\noindent\textbf{[14]}
X = \{0, 1, 2, 3, 4, 5, 6, 8, 9, 10, 14, 16, 17, 18, 20, 22, 23, 27, 32, 33, 34, 37, 39, 41, 44, 45, 47, 52, 54, 55, 57, 58, 60, 62, 65, 67, 71, 76, 79\} \\
Y = \{0, 1, 2, 3, 8, 12, 13, 14, 18, 19, 20, 22, 23, 26, 27, 29, 30, 35, 38, 39, 42, 44, 48, 51, 58, 59, 60, 62, 71, 73, 78, 79, 81, 82\}
\end{multicols}

\normalsize

\section{A reflection criterion on the noncyclic group of order 99}
\label{app:reflection}

On the noncyclic group of order $99$, reflection makes each compressed
entry determine its binary fibre uniquely. The channel criterion can
therefore test this restricted family directly through its compressions.
Let $G=\Z_3\times\Z_3\times\Z_{11}$ and
$r(x,y,z)=(x,-y,z)$. For an $r$-invariant binary sequence write
$a_0(x,z)=a(x,0,z)$, $a_1(x,z)=a(x,1,z)=a(x,-1,z)$ and
\[
 U_a=a_0+2a_1,\qquad T_a=(a_0-a_1)/2,
\]
with analogous notation for $b$. These functions are on
$\Z_3\times\Z_{11}\cong\Z_{33}$. The sequences $T_a,T_b$ are ternary,
meaning that their entries lie in $\{0,\pm1\}$. A weighing matrix $W(n,k)$
is a matrix $W$ of order $n$ with entries in $\{0,\pm1\}$ and
$WW^{\top}=kI_n$. The next result identifies such a matrix as a necessary
part of the construction and gives the further constraints linking it
to the compressed signs.

\begin{proposition}
\label{prop:reflection_appendix}
The map $a\mapsto U_a$ is a bijection between $r$-invariant binary sequences
and functions with values in $\{\pm1,\pm3\}$. An invariant pair with sums
$13,15$ is D-optimal if and only if
\[
 \PAF_{U_a}+\PAF_{U_b}=6\quad\text{off zero},\qquad
 \PAF_{T_a}+\PAF_{T_b}=49\delta_0.
\]
In this case $\sum T_a=-7$, $\sum T_b=0$, and their group matrices give
the weighing matrix
$W=\left(\begin{smallmatrix}C_{T_a}&C_{T_b}\\
 -C_{T_b}^{\top}&C_{T_a}^{\top}\end{smallmatrix}\right)$,
with $WW^{\top}=49I_{66}$, where $C_T=(T(h-g))_{g,h\in\Z_3\times\Z_{11}}$.
Furthermore, let $S_s$ give the sign of each entry of $U_s$, for $s=a,b$, and put
\[
 P=\PAF_{S_a}+\PAF_{S_b},\qquad
 C(k)=\sum_q\sum_{s\in\{a,b\}}
       \bigl(S_s(q)T_s(q+k)+T_s(q)S_s(q+k)\bigr).
\]
Then $U_s=3S_s+2T_s$, $T_s=(U_s^3-9U_s)/8$, and
\[
 \sum S_a=9,\quad \sum S_b=5,\quad C(0)=-98,\qquad
 3P(k)+2C(k)=2\quad(k\ne0).
\]
The supports, or sets of nonzero positions, of $T_a,T_b$ have odd and even sizes respectively, summing
to $49$, and $C(k)\equiv4\pmod6$ for $k\ne0$.
\end{proposition}

\begin{proof}
The four possibilities for $(a_0,a_1)$ give the bijection and the
pointwise formulas. Characters trivial on the $y$-coordinate see $U_a$;
the other characters see $a_0-a_1=2T_a$, since
$\zeta_3+\zeta_3^{-1}=-1$. The two autocorrelation conditions therefore
give, and are necessary for, spectral sum $196$ at every nontrivial
character of $G$. The assumed sums supply the trivial character.

To determine the sums of $T_a,T_b$, compress onto the $y$-coordinate.
The two symmetric triples have the form $(13-2u,u,u)$ and
$(15-2v,v,v)$, with $u,v$ odd, and their combined energy is $262$.
Consequently
$((3u-13)/2)^2+((3v-15)/2)^2=49$.
The integer representations of $49$ as two squares, with the two
coordinates congruent to $1,0$ modulo three, give $u=9,v=5$.
Thus $\sum T_a=(-5-9)/2=-7$ and $\sum T_b=0$.
The group matrices commute with one another and their transposes, so the
autocorrelation identity for $T$
gives $WW^{\top}=49I_{66}$.

The sums of $S$ follow from $U=3S+2T$. A ternary sequence's support
size has the parity of its sum; the combined support size is its
energy $49$. Also $S_sT_s=-1$ on the support of $T_s$, giving $C(0)=-98$.
Expanding the autocorrelation of $U=3S+2T$ gives
$9P+4(\PAF_{T_a}+\PAF_{T_b})+6C$, hence $3P+2C=2$ off zero.
Since every binary autocorrelation of length $33$ is $1$ modulo four,
$P\equiv2\pmod4$ and $C\equiv4\pmod6$ there.
\end{proof}

The weighing-matrix condition must be coupled to the compressed signs
through these identities. An arbitrary $W(66,49)$ does not supply a
reflection-invariant D-optimal pair, and the criterion alone does not
exclude such pairs.

\end{document}